\documentclass[reqno,11pt]{amsart}
\usepackage{color}
\usepackage{amsmath}
\usepackage{amsfonts}
\usepackage{amssymb}
\usepackage{amsthm}
\usepackage{newlfont}
\usepackage{graphicx}

\usepackage{geometry}
\newtheorem{thm}{Theorem}[section]

\newtheorem{lem}{Lemma}[section]
\theoremstyle{remark}
\newtheorem{rem}{Remark}[section]
\newtheorem*{rem*}{Remark}

\newcommand{\ed}{\end {document}}

\newcounter{smalllist}

\title[Counterexample for higher order]{Remarks on the Bernstein inequality  for higher
order operators and related results}
\author[D. Li]{Dong Li}
\address{D. Li, SUSTech International Center for Mathematics, and Department of Mathematics,  Southern University of Science and Technology, Shenzhen, China}%
\email{lid@sustech.edu.cn}
\author[Y. Sire]{Yannick Sire}
\address{Y. Sire, Department of Mathematics, Johns Hopkins University, Baltimore, MD 21218, USA}
\email{ysire1@jhu.edu}
\begin{document}

\begin{abstract}
This note is devoted to several results about frequency localized functions and associated Bernstein inequalities for higher order operators. In particular,  we construct some counterexamples for the frequency-localized Bernstein inequalities for higher order Laplacians. We show also that the heat semi-group associated to powers larger than one of the laplacian  does not satisfy the strict maximum principle in general. Finally, in a suitable range we provide several positive results.
\end{abstract}
\maketitle

\tableofcontents

\section{Introduction}
This note is devoted to several results about frequency localized functions and
associated Bernstein inequalities for higher order operators. We consider a class
of fractional Laplacian operators acting on frequency localized functions on the whole space $\mathbb R^d$ or the periodic torus.
 To fix the notation, we use the following convention for Fourier
transform on $\mathbb R^d$, $d\ge 1$:
\begin{align}
& f(x) = \frac 1 {(2\pi)^d} \int_{\mathbb R^d} \widehat f(\xi) e^{i \xi \cdot x} d \xi;\\
& \widehat f(\xi) = \int_{\mathbb R^d} f(x) e^{- i \xi \cdot x } dx.
\end{align}
For $s>0$, we define the fractional Laplacian operator $\Lambda^s = (-\Delta)^{\frac s2}$ via
the Fourier transform:
\begin{align}
\widehat{ \Lambda^s f } (\xi) = |\xi|^s \widehat f (\xi), \qquad \xi \in \mathbb R^d.
\end{align}
In yet other words $\Lambda^s$ corresponds to the Fourier multiplier $|\xi|^s$.  Note that
for $s=2$ we have $ -\Lambda^s = \Delta$, i.e. the usual Laplacian operator.  For $0<s\le 2$,
it is known (cf. \cite{CL12}, \cite{L12}, \cite{Wu1} and the references therein) that the following 
frequency-localized 
Bernstein-type inequality hold:  for $1<p<\infty$ and any band-limited
 $f \in L^p(\mathbb R^d, \mathbb R)$ with
\begin{align} \label{c1}
\text{supp}(\widehat f) \subset \{ \xi:\;
\gamma_1 \le |\xi| \le \gamma_2\},
\end{align}
 there are  constants $A_1>0$, $A_2>0$ depending only on $(d,p, s, \gamma_1,\gamma_2)$ such that
\begin{align}
A_2 \| f\|_{p}^p\le \int_{\mathbb R^d} (\Lambda^{s} f) |f|^{p-2} f dx \le A_1  \| f\|_{p}^p
=A_1 \int_{\mathbb R^d} |f|^p dx.
\label{Bern1}
\end{align}
Note that for $p=2$, the above inequality is trivial thanks to the usual Plancherel theorem. 
The main point of \eqref{Bern1} is that it continues to hold for $p\ne 2$ where the Fourier support
of the associated functions  have nontrivial overlapping interactions.

By a scaling argument, if $h\in \mathcal S(\mathbb R^d)$ has frequency localized into
$\{ |\xi| \sim N \}$ where $N\gg 1$, then it follows from 
\eqref{Bern1} that (below $0<s\le 2$ and $1<p<\infty$)
\begin{align}
\int_{\mathbb R^d} \Lambda^s h  |h|^{p-2} h dx \ge \; \mathrm{const} \cdot  N^s \| h \|_p^p.
\end{align}
Such powerful estimates have important applications in the regularity theory of fluid dynamics
equations (cf. \cite{Wu1}).  For example, consider the dissipative two-dimensional surface quasi-geostrophic equation 
\begin{align}
\partial_t \theta = - \Lambda^s \theta + \Lambda^{-1} \nabla^{\perp} \theta \cdot \nabla \theta,
\end{align}
where $0<s\le 2$. 
Applying  the Littlewood-Paley project $P_j$ which is localized to $\{|\xi| \sim 2^j \}$ and calculating
the $L^p$ norm of $P_j \theta$, we obtain
\begin{align}
\frac 1 p \partial_t ( \| P_j \theta \|_p^p)
&= - \int_{\mathbb R^d} (\Lambda^s P_j \theta) |P_j \theta|^{p-2} P_j \theta dx +
\text{Nonlinear terms} \\
&\le - \mathrm{const} \cdot 2^{js} \| P_j \theta \|_p^p +
\text{Nonlinear terms}, \qquad (\text{by Bernstein}).
\end{align}
From this and using additional (nontrivial) commutator estimates, one can deduce fine regularity 
results in various critical and subcritical Besov spaces (see recent \cite{Li21} for an optimal Gevrey
regularity result and the references therein for earlier results). On the other hand, it has been long speculated\footnote{We would
like to thank Professor Jiahong Wu for raising this intriguing question.}
 whether the above Bernstein inequalities also hold for higher order Laplacian operators 
 $\Lambda^s$ for $s>2$. The purpose of this note is to demonstrate some counterexamples
 around these higher operators $\Lambda^s$.   Our main results are the following.
 
 \begin{itemize}
 \item Biharmonic operator.  See Theorem \ref{t1}. We show via an explicit construction the
 failure of Bernstein inequalities for the biharmonic operator $\Delta^2$ with $p=4$.
 \item Lack of positivity for higher order $e^{-\Lambda^s}\delta_0$, $s>2$. See Theorem \ref{thm2.1}.
 We give two proofs to show the general lack of positivity for the higher order heat
 operators. Some sharp asymptotic decay at spatial infinity is also shown. 
 \item Counterexamples for Bernstein for $s>2$, $p\in (1, p_0)$ or $p \in (p_1, \infty)$ for
 some $p_0<2<p_1$.  See Theorem \ref{thm2.2} and Theorem \ref{thm3.3}.  For general operators $\Lambda^s$ with $s>2$, we show
 generic failure of Bernstein inequalities for  $p=1+$ or $p=\infty-$. 
 \item Some periodic Bernstein inequalities for $\Lambda^s$, $0<s\le 2$.  See Theorem \ref{thm3.1}
 and Theorem \ref{thm4.2a}.
 By using a nontrivial complex interpolation argument together with some concentration 
 inequality, we prove a family of Bernstein inequalities for mean-zero periodic functions for all
 $p\in (1,\infty)$. We also show frequency-localized versions in Theorem \ref{thm4.2a}.
 \item A Liouville theorem for $\Lambda^s$, $s>0$. See Theorem \ref{thm5.1}. We prove a rigidity
 type for the ancient solutions to a fractional heat equation. 
 \end{itemize}
 
 The rest of this note is organized according to the above summary. 
 
 \subsection*{Notation} 
  For any two positive quantities $X$ and $Y$, we write $X \lesssim Y$ or $X=O(Y)$ if 
 $X\le C Y$ for some unimportant constant $C>0$.  We write $X \ll Y$ if 
 $X \le c Y$ for some sufficiently small constant $c>0$.  The needed smallness is clear
 from the context. We write $f \in L^p(\Omega, Y)$ if $f:\; \Omega \to Y$ and is in $L^p$. For example $f \in L^2(\mathbb R^3, \mathbb R^2)$ means $f$ is $\mathbb R^2$-valued and
 \begin{align}
 \| f \|_2^2 = \int_{\mathbb R^3} |f|^2 dx = \int_{\mathbb R^3}
 (f_1^2 +f_2^2) dx <\infty, \qquad \text{here $f=(f_1, f_2)^{\mathrm T} $}.
 \end{align}
 
For a complex number $z=a+bi$ with $a, b \in \mathbb R$, we denote $\mathrm{Re}(z)=a$
and $\mathrm{Im}(z) = b$. 

We denote the usual sign function $\mathrm{sgn}(x) = 1$ for $x>0$, $-1$ for $x<0$ and $0$ if $x=0$.

We use the Japanese bracket notation $\langle x \rangle = \sqrt{1+|x|^2}$ for any
$x \in \mathbb R^d$, $d\ge 1$.
 
 \section{Failure of Bernstein inequalities}

\begin{thm} \label{t1}
Let the dimension $d\ge 1$. 
 There exists a sequence of Schwartz functions
$f_j:\; \mathbb R^d \to \mathbb R$ with frequency localized around $N_j \to \infty$, such that 
\begin{align*}
-C_2< \frac{\int_{\mathbb R^d } \Delta^2  f_j f_j^3 dx } { N_j^4 \int_{\mathbb R^d} f_j^4 dx } <-C_1<0.
\end{align*}
In the above $C_1>0$, $C_2>0$ are constants depending only on $d$. More precisely, the frequency support of $f_j$
satisfies
\begin{align*}
\mathrm{supp}(\widehat{f_j} ) \subset \{ \xi:\,  \alpha_1 N_j < |\xi| < \alpha_2 N_j \},
\end{align*}
where $\alpha_1>0$, $\alpha_2>0$ are constants depending only on $d$. 
\end{thm}
\begin{rem}
By a perturbative argument, one can also show counterexamples for $\Lambda^s$,
$|s-4| \ll1$. 
\end{rem}

\begin{lem} \label{a1}
Consider $f_1(x) = \log (1+x^2)$, $x \in \mathbb R$. Denote $f_1^{(4)} (x)
= \frac {d^4} {dx^4} ( f_1(x) )$. Then 
\begin{align*}
I_1= \int_{\mathbb R} f_1^{(4)}(x)  f_1(x)^3 dx <0.
\end{align*}
\end{lem}
\begin{rem}
One can take for example $f_2(x)= x+e^{-x^2}$ to obtain 
\begin{align}
\int f_2^{(4)}(x) (f_2(x))^3 dx \approx -2.47784<0.
\end{align}
However the issue with $f_2$ is that it is not amenable to localization. Namely if we consider
\begin{align}
\int (f_2(x) \phi(\frac x R) )^{(4)} ( f_2(x) \phi(\frac x R ) )^3
\end{align}
for a bump function $\phi$ and $R$ large, then the main order term is
\begin{align}
\int ( x \phi(\frac x R) )^{(4)} ( x \phi(\frac x R ) )^3 dx
=R \int ( x \phi(x) )^{(4)} ( x \phi (x) )^3 dx 
\end{align} 
which may not take a favorable sign. This subtle issue disappears for the function
$f_1(x)=\log (1+x^2)$ due to its mild growth at spatial infinity. 
\end{rem}
\begin{rem}
Interestingly, if we work with $\frac {d^8} {dx^8}$ instead of $\frac {d^4} {dx^4}$, then
we have for $f_2(x) = x+e^{-x^2}$, 
\begin{align}
\int f_2^{(8)} (x) (f_2(x) )^3 dx \approx -219.804<0.
\end{align}
One may then take $f_{2,R}(x) = f_2(x) \phi(x/R)$ for $R$ sufficiently large to show
\begin{align}
\int f_{2,R}^{(8)} (x) (f_{2,R}(x) )^3 dx <0.
\end{align}
This can be used to construct frequency localized counterexamples for $
\Lambda^8=(-\partial_{xx})^4$.
\end{rem}

\begin{rem}
To obtain $I_1<0$, we can also adopt a more \emph{numerical} approach in lieu of
exact contour integral computation.  To this end, denote 
\begin{align}
g(x)= f_1^{(4)}(x) f_1(x)^3 = - 12 \frac {(1-6x^2+x^4) (\log (1+x^2) )^3}
{(1+x^2)^4}.
\end{align}
A schematic drawing of $g(x)$ for $x\in [0, 0.5]$ and $x \in [0,10]$ can be found
in the figures below.
\begin{figure}[!h]
\includegraphics[width=0.4\textwidth]{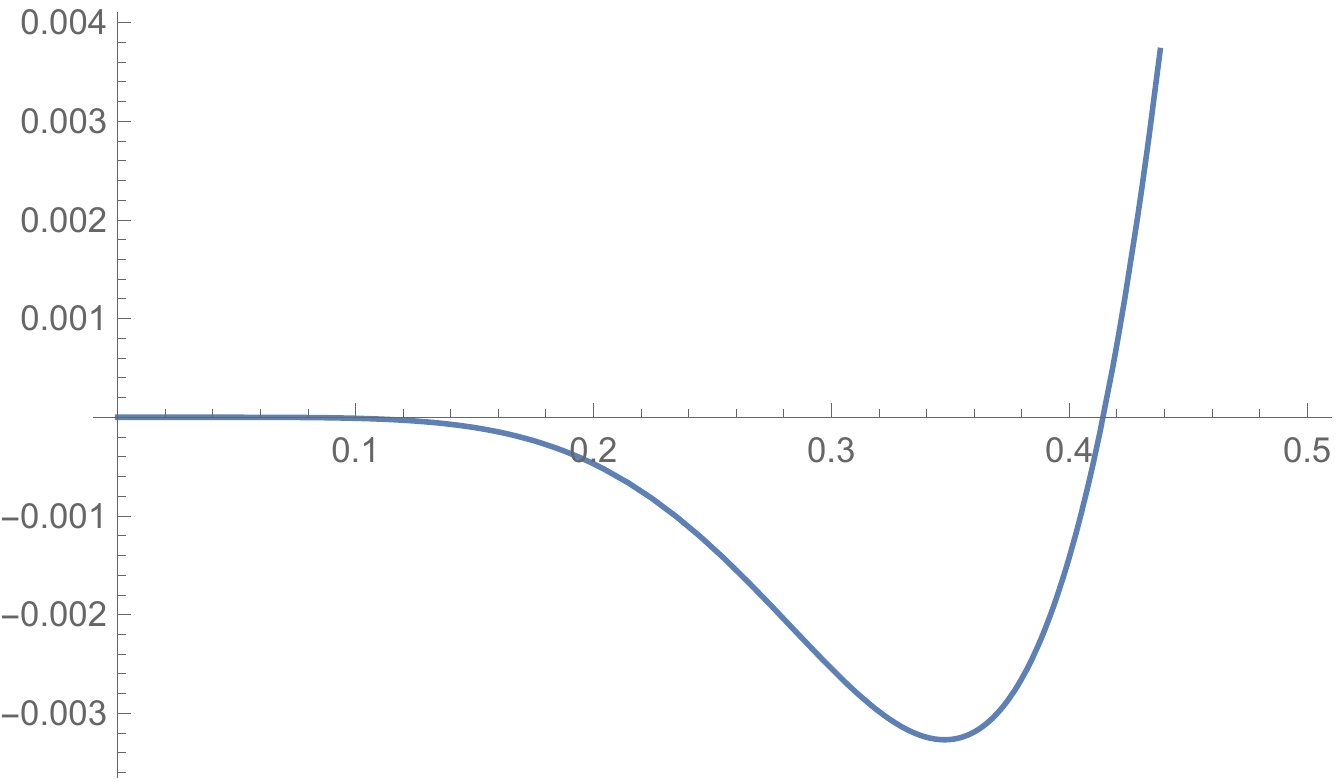}
\includegraphics[width=0.4\textwidth]{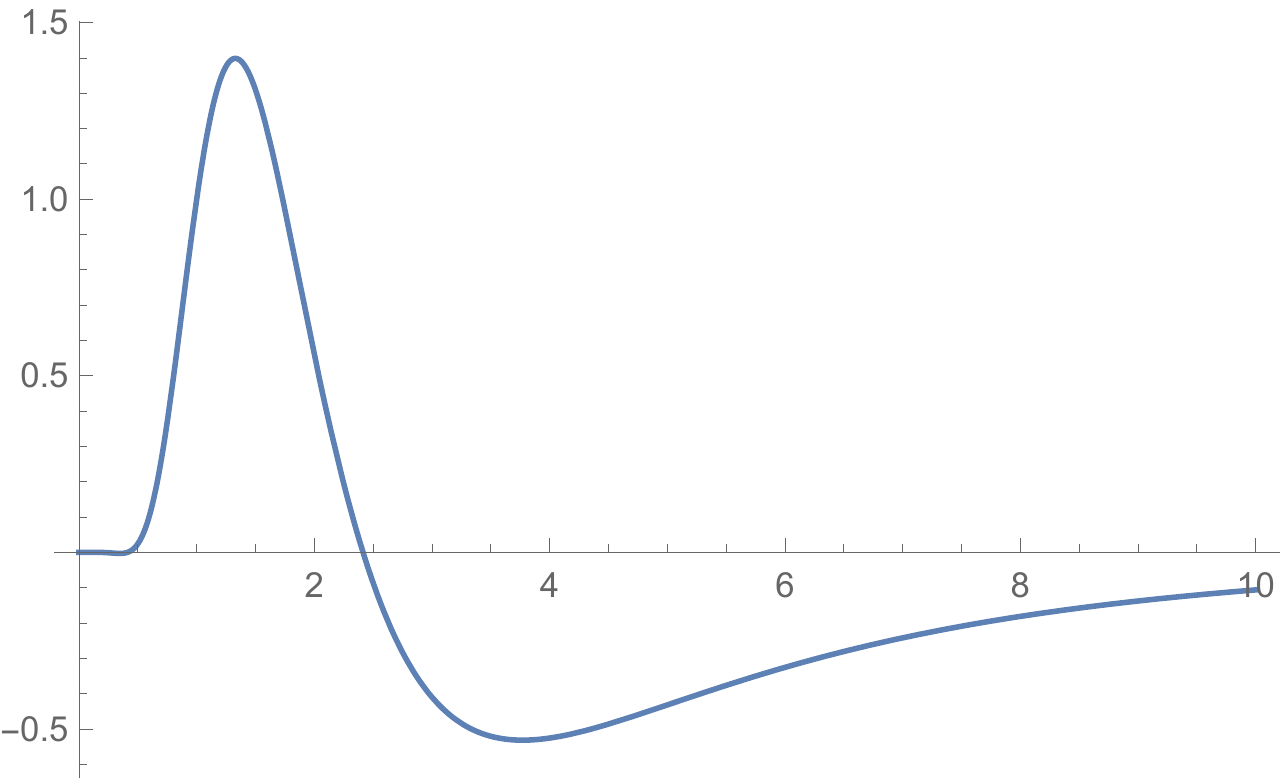}
\caption{The plot of $g(x)$ for $0\le x \le 0.4$  and $0\le x \le 10$}\label{fig:0}
\end{figure}
By examining the polynomial $1-6x^2+x^4$ in the definition of $g(x)$, it is
easy to check that $g(x)>0$ for $x\in (\sqrt 2 -1, \sqrt 2 +1)$ and $g(x)<0$
for $x<\sqrt 2-1$ or $x>\sqrt 2+1$.  In particular
\begin{align}
I_1 <2 \int_{0\le x \le 10} g(x) dx \approx -1.65835.
\end{align}

\end{rem}

\begin{proof}
To ease the notation we write $f_1$ as $f$ and $\int_{\mathbb R} dx $ as $\int$.  By successive integration by parts, we have
\begin{align*}
I_1&= \int f^{\prime\prime} (f^3)^{\prime\prime} = \int f^{\prime\prime} (3f^2 f^{\prime})^{\prime} \\
& = 3 \int f^2 (f^{\prime\prime})^2 + \int 6 f^{\prime\prime} (f^{\prime})^2 f \\
& = 3 \int f^2 (f^{\prime\prime})^2 -2 \int (f^{\prime})^4.
\end{align*}
Note that $f^{\prime}(x)=2x/(1+x^2)$. By a contour integral computation, it is not difficult to check
that 
\begin{align*}
2 \int (f^{\prime})^4 = 2\pi.
\end{align*}
On the other hand (see Appendix A), we have
\begin{align}  \label{2.11A}
3 \int f^2 (f^{\prime\prime})^2 =
-\frac {29}6 \pi + \pi^3 + (\log 4) (-7+ \log 64) \pi.
\end{align}
Thus 
\begin{align*}
I_1 =  -\frac {41}6 \pi + \pi^3 + (\log 4)  (-7 +\log 64) \pi \approx -2.83.
\end{align*}
\end{proof}
We now complete the proof of Theorem \ref{t1}.
\begin{proof}[Proof of Theorem \ref{t1}]
We proceed in several steps.

Step 1. We first construct $f_j$ in the one dimensional case. 
To ease the notation we shall denote
\begin{align*}
I(f) =  \int_{\mathbb R} \partial_x^4 f f^3 dx.
\end{align*}
We choose $f_1$ as in Lemma \ref{a1}.  Clearly
\begin{align*}
I(f_1) = \int_{\mathbb R} \partial_x^4 f_1 f_1^3 dx \approx -2.83 <0.
\end{align*}
Define  for $R\ge 2$
\begin{align*}
f_{R}(x) = f_1(x) \phi( x /R) = \log(1+x^2) \phi( x/ R),
\end{align*}
where $\phi \in C_c^{\infty}(\mathbb R)$ is such that $\phi(z)=1$ for $|z|  \le 1$ and 
$\phi(z)=0$ for $|z| \ge 2$. By taking $R$ sufficiently
large, it is not difficult to check that $I(f_R)<0$.
As a matter of fact $I(f_R)\to I(f_1)$ as $R \to \infty$.  We now fix $R=R_0$ such that $I_{R_0}<0$. 
Clearly $f_{R_0} \in C_c^{\infty}(\mathbb R)$. 

Next we take $\epsilon>0$ and define $h_{\epsilon} \in \mathcal S(\mathbb R)$ such that
\begin{align*}
\widehat{h_{\epsilon}} (\xi) = \phi(\epsilon \xi) 
\left( 1-  \phi \Bigl( \frac {\xi} {\epsilon} \Bigr)  \right) \widehat{f_{R_0}}
(\xi), \qquad \xi \in \mathbb R.
\end{align*}
Clearly  $I(h_{\epsilon}) \to I(f_{R_0})$ as $\epsilon \to 0$. Thus we can fix $\epsilon_0>0$ sufficiently small such that $I(h_{\epsilon_0} ) <0$. 

Finally we define $f_j \in \mathcal S(\mathbb R)$ such that
\begin{align*}
\widehat{f_j}(\xi) = N_j^{-\frac 34} \widehat{h_{\epsilon_0} } \Bigl( \frac {\xi} {N_j} \Bigr).
\end{align*}
On the real side, we have
\begin{align*}
f_j(x) = N_j^{\frac 14}  h_{\epsilon_0} ( N_j x).
\end{align*}
Apparently $\|f_j \|_4  = \| h_{\epsilon_0} \|_4$ for all $j$.   Clearly
 $f_j$ satisfies the desired constraints in dimension $d=1$.

Step 2. Higher dimensions.  With no loss we consider dimension $d=2$. The case for $d\ge 3$ is 
similar and omitted.  Define
\begin{align*}
f_j (x_1, x_2) = N_j^{\frac 14} h_{\epsilon_0} (N_j x_1) \psi (x_2),
\end{align*}
where $h_{\epsilon_0}$ was specified in Step 1, and $\psi\in \mathcal S(\mathbb R)$ is chosen
to have frequency localized to $\frac 12 \le |\xi| \le 1$.  Clearly
\begin{align*}
\int_{\mathbb R^2} \Delta^2 f_j f_j^3 dx 
& = \int_{\mathbb R^2} \partial_{x_1}^4 f_j  f_j^3 dx  + \mathrm{l.o.t.} 
\end{align*}
The desired conclusion follows easily.
\end{proof}

Consider $s>2$, and fix any $p\ne 2$. A general question  is whether one 
can find smooth frequency localized $f$ such that
\begin{align*}
\int_{\mathbb R^d} (-\Delta)^{\frac s2}  f |f|^{p-2} f dx <0. 
\end{align*}
All these have deep connections with the lack of positivity of the fundamental solution
for higher order heat propagators. In the next section we investigate  somewhat more
general situation concerning $\Lambda^s$, $s>2$. 

\section{Lack of positivity for $e^{-\Lambda^s} \delta_0$, $s>2$ }

\begin{lem}[\cite{Po}] \label{lem_s1}
Let $\alpha>0$. Define
\begin{align*}
F_{\alpha}(x)  = \int_0^{\infty} e^{-t^{\alpha}} \cos x t dt, \quad x>0.
\end{align*}
Then
\begin{align*}
\lim_{x\to \infty} x^{\alpha+1} F_{\alpha}(x) = \Gamma (\alpha+1) \sin \frac {\pi \alpha}2.
\end{align*}
\end{lem}
\begin{figure}[!h]
\includegraphics[width=0.4\textwidth]{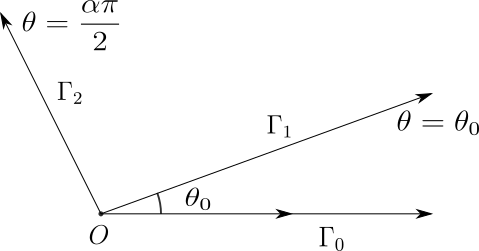}
\caption{Contours $\Gamma_1$ and $\Gamma_2$ }\label{fig:1}
\end{figure}

\begin{proof}
We briefly recall the argument of Polya as follows. First by using partial integration one has
\begin{align*}
x^{\alpha+1} F_{\alpha}(x) & =  x^{\alpha} 
\int_0^{\infty} (\sin x t) e^{-t^{\alpha}} d (t^{\alpha} ) \\
& = \int_0^{\infty} \sin  u^{\frac 1{\alpha}  }  e^{-x^{-\alpha} u} du \qquad( 
\text{ $u=t^{\alpha} x^{\alpha}  $ } )\\
& = \mathrm{Im} \Bigl( \int_0^{\infty} e^{i u^{\frac 1{\alpha} } - x^{-\alpha} u } du \Bigr) . 
\end{align*}
Now first deform the contour $\Gamma_0=[0, \infty)$ (see Figure \ref{fig:1}) to $\Gamma_{1}= \{ u = r e^{i\theta_0}:\; 0\le r <\infty \}$ 
for $0<\theta_0 \ll 1$, one has\footnote{This step is necessary 
since the integrand contains $e^{-x^{-\alpha} u}
=e^{-x^{-\alpha} (r \cos \theta+ i r \sin \theta) }$ for $u=r e^{i\theta}$, and 
$r \cos \theta$ may become negative 
if $\theta$ goes from $0$ to $\frac {\alpha \pi} 2$ especially when $\alpha>1$. }
\begin{align*}
\lim_{x\to \infty} x^{\alpha+1} F_{\alpha}(x) =
\mathrm{Im} \Bigl( \int_{\Gamma_1} e^{i u^{\frac 1 {\alpha} } } du \Bigr).
\end{align*}
One can then deform the latter integral from $\Gamma_1$ to
$\Gamma_2= \{ u=r e^{i \frac {\alpha \pi} 2 }:\, 0\le r <\infty\}$ to obtain
\begin{align*}
\mathrm{Im} \Bigl( \int_{\Gamma_1} e^{i u^{\frac 1 {\alpha} } } du \Bigr)
=\sin(\frac {\pi \alpha} 2)  \int_0^{\infty} e^{-r^{\frac 1 {\alpha} } } dr
= \Gamma(\alpha+1) \sin \frac {\pi \alpha} 2.
\end{align*}
\end{proof}
\begin{rem}
We shall need to use the standard Bessel functions: for $\nu>0$ and 
$z = \rho e^{i\theta}$ with $-\pi <\theta \le \pi$, 
\begin{align*}
J_{\nu}(z) = C_{\nu} z^{\nu} \int_{-1}^1 (1-s^2)^{\nu -\frac 12} e^{is z} ds,
\end{align*}
where $C_{\nu}>0$ depends only on $\nu$.  In particular we recall the usual formula
for Bessel functions (cf. pp. 11 of \cite{Er}): 
\begin{align} \label{e2.32}
\frac d{dz}( J_{\nu+1}(z) z^{\nu +1}) =
J_{\nu}(z) z^{\nu+1}.
\end{align}
We also recall (cf. pp. 168 of \cite{Wat} and pp. 149 of \cite{Leb}): for $\mathrm{Re}(\nu+\frac 12)>0$, $-\frac 12 \pi <\mathrm{arg}(z)<
\frac 32 \pi$, 
\begin{align} \label{e2.33}
H_{\nu}^{(1)} (z) 
=\left( \frac 2 {\pi z } \right)^{\frac 12} 
\frac 1 {\Gamma(\nu+\frac 12)}
e^{ i (z -\frac 12 \nu \pi -\frac 14 \pi )} 
\int_0^{\infty} e^{- u} u^{\nu -\frac 12}
(1+\frac {i u } {2 z} )^{\nu -\frac 12} du.
\end{align}
\end{rem}
\begin{lem}[\cite{BG}] \label{lem_s2}
Consider $d\ge 2$. Let $\alpha>0$ and define 
\begin{align}
F_{\alpha}(x) = \int_{\mathbb R^d} e^{- |\xi|^{\alpha} } e^{i x \cdot \xi} d\xi, \qquad x \in \mathbb R^d.
\end{align}
Then 
\begin{align}
\lim_{|x| \to \infty} |x|^{d+\alpha} F_{\alpha}(x) = C_{d,\alpha} \sin \frac {\alpha \pi}2,
\end{align}
where $C_{d,\alpha}>0$ depends only on ($d$, $\alpha$). 
\end{lem}
\begin{figure}[!h]
\includegraphics[width=0.4\textwidth]{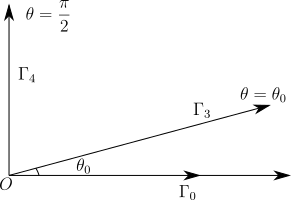}
\caption{Contours $\Gamma_3$ and $\Gamma_4$ }\label{fig:3}
\end{figure}

\begin{proof}
To simplify the notation we shall denote by $C$ a positive constant depending only
on ($d$, $\alpha$) which may vary from line to line. Denote $r=|x|$ and $t = |\xi|$. By passing
to hyper-spherical coordinates, we have
\begin{align*}
r^{d+\alpha} F_{\alpha}(x) &= C r^{d+\alpha} \int_0^{\infty} e^{-t^{\alpha}} t^{d-1}  \int_0^1
(1-s^2)^{\frac {d-3} 2}  \cos (rts) \; ds dt \\
&= C r^{\alpha} \int_0^{\infty} e^{- r^{-\alpha} t^{\alpha} }  t^{d-1}
\int_0^1 (1-s^2)^{\frac {d-3}2} \cos (ts) ds dt  \qquad (\text{$rt \to t$}).
\end{align*}
It is not difficult to check that (see \eqref{e2.32}, or one can verify directly the computation)
\begin{align*}
\frac d {dt} \Bigl( t^{d}
\int_0^1 (1-s^2)^{\frac {d-1}2} \cos (ts) ds \Bigr)
= C t^{d-1}
\int_0^1 (1-s^2)^{\frac {d-3}2} \cos (ts) ds.
\end{align*}
Thus
\begin{align*}
r^{d+\alpha} F_{\alpha}(x) &= C
\int_0^{\infty} e^{- r^{-\alpha} t^{\alpha} } 
t^{d+\alpha-1} \int_0^1 (1-s^2)^{\frac {d-1}2 }
\cos (ts) ds dt \\
& = C  \int_0^{\infty} e^{- r^{-\alpha} t^{\alpha}}
t^{\frac d2 +\alpha-1} J_{\frac d2} (t) dt =
C \mathrm{Re} \left(  \int_0^{\infty} e^{- r^{-\alpha} t^{\alpha}}
t^{\frac d2 +\alpha-1} H^{(1)}_{\frac d2} (t) dt \right).
\end{align*}
By \eqref{e2.33}, it suffices for us to examine (below $0<\theta_0\ll 1$ is a fixed angle)
\begin{align*}
&\lim_{\epsilon\to 0 } \mathrm{Re} \left( \int_0^{\infty} e^{- \epsilon t^{\alpha} } 
t^{\frac d2+ \alpha-1} 
\Bigl( t^{-\frac 12} e^{i t - i\frac {d+1} 4 \pi}
\int_0^{\infty} e^{-s} s^{\frac {n-1}2}
( 1+ \frac {is} {2t } )^{\frac {d-1} 2} ds \Bigr)
d t \right)  \\
=& \mathrm{Re} 
\left( \int_{\Gamma_3} z^{\frac {d} 2 +\alpha-\frac 32}e^{iz- i \frac {d+1}4 \pi} 
\Bigl( \int_0^{\infty} e^{-s} s^{\frac {d-1}2} ( 1+ \frac {is}{2z})^{\frac {d-1}2 } ds \Bigr) dz
\right)   \qquad ( \text{ $ \Gamma_3: \{z=r e^{i\theta_0} :\; 0\le r <\infty \}$})\\
= & \mathrm{Re} 
\left( \int_{\Gamma_4} z^{\frac {d} 2 +\alpha-\frac 32}e^{iz- i \frac {d+1}4 \pi} 
\Bigl( \int_0^{\infty} e^{-s} s^{\frac {d-1}2} ( 1+ \frac {is}{2z})^{\frac {d-1}2 } ds \Bigr) dz
\right)  \quad ( \text{ $\Gamma_4: \{ z=\rho i: \; 0\le \rho <\infty\}$})\\
=& (\sin \frac {\alpha \pi }2) \int_0^{\infty}
\rho^{\frac d2 +\alpha-\frac 32} e^{-\rho}
\Bigl( \int_0^{\infty} e^{-s} s^{\frac {d-1}2}
( 1+ \frac s {2\rho} )^{\frac {d-1} 2} d s  \Bigr) d \rho.
\end{align*}
The desired result clearly follows.

\end{proof}

\begin{thm}[Lack of positivity for the propagator $e^{- \Lambda^s}$ when $s>2$]
\label{thm2.1}
Define  for $s>0$, 
\begin{align*}
K_s(x) = \frac 1 {2\pi} \int_{\mathbb R} e^{-|\xi|^s} e^{i \xi \cdot x} d\xi.
\end{align*}
If $s>2$, then 
\begin{align*}
\min_{x\in \mathbb R} K_s(x) <0.
\end{align*}
More generally define 
\begin{align*}
K_{s,d}(x) = \frac 1 {(2\pi)^d} \int_{\mathbb R^d} e^{-|\xi|^s} e^{ i \xi \cdot x } d\xi.
\end{align*}
If $s>2$, then 
\begin{align*}
\min_{x\in \mathbb R^d} K_{s,d}(x) <0.
\end{align*}
\end{thm}
\begin{proof}
We first consider the 1D case.
By Lemma \ref{lem_s1}, it is not difficult to check that $K_s(x) <0$ for $2<s<4$. 
We claim that for any $s\ge 4$, we must have $\inf K_s (x) <0$. Assume this is not true and
for some $s_0\ge 4$, it holds that $ K_{s_0}(\cdot)$ is always nonnegative. 
By using the usual subordination principle, for any $\beta \in (0,1)$, $t>0$, it holds
that 
\begin{align*}
e^{-t^{\beta} } = \int_0^{\infty} e^{-\lambda t} d \mu_{\beta} (\lambda),
\end{align*}
where $d\mu_{\beta}$ is a positive measure.  Taking $\beta_0 \in (0,1)$ sufficiently small,
we have
\begin{align*}
e^{-|\xi|^{s_0\beta_0} }= \int_0^{\infty} e^{-\lambda |\xi|^{s_0} } d \mu_{\beta_0} (\lambda),
\end{align*}
where $s_0 \beta_0 \in (2,4)$.  But then it follows that $K_{s_0\beta_0}$ must be nonnegative.
This is clearly a contradiction. This finishes the proof for the 1D case.
The higher dimensional case is similar by using Lemma \ref{lem_s2}.
\end{proof}

We now give yet another proof of Theorem \ref{thm2.1} based on a contradiction argument.
We first recall the usual Bochner theorem: namely if $F(\xi) = \mathbb E e^{-i \xi \cdot x}$ 
($\mathbb E(\cdot)$ denotes taking expectation with respect to some probability measure 
on $\mathbb R^n$), then $F(\cdot)$ must be a positive definite function. In particular we must
have
\begin{align}
|F(\xi)| \le |F(0)|, \qquad \forall\, \xi \in \mathbb R^d.
\end{align}
With this we now give an alternative proof of Theorem \ref{thm2.1}.
\begin{proof}[\underline{$2^{\mathrm{nd}}$ proof of Theorem \ref{thm2.1}}]
We argue by contradiction. Assume that 
\begin{align}
e^{-|\xi|^s}= \int_{\mathbb R^d} f(x) e^{- i x \cdot \xi} dx,
\end{align}
where $f$ is nonnegative for all $x \in \mathbb R^d$. We shall deduce a contradiction.

By using Fourier transform it is not difficult to check that $|\cdot|^2 f(\cdot) \in L_x^1(\mathbb R^d)$.  In particular we have
\begin{align}
\tilde F(\xi) = -\Delta_{\xi} (e^{-|\xi|^s}  ) = \int_{\mathbb R^d}
f(x) |x|^2 e^{-i x \cdot \xi} dx
\end{align}
and $\tilde F(\cdot)$ is continuous and positive definite.  Thus we must have
\begin{align}
|\tilde F(\xi) | \le |\tilde F(0)|, \qquad \forall\, \xi \in \mathbb R^d.
\end{align}
However since $s>2$, it is easy to check that $\tilde F(0) =0$ which clearly gives a contradiction!
\end{proof}

We now draw some consequences of the previous theorem.

\begin{thm} \label{thm2.2}
Let the dimension $d\ge 1$.  Let $s>2$.  There exists $p_0>2$ depending only on
($s$, $d$) such that for any $p \in [p_0, \, \infty)$, we can find $f \in C_c^{\infty}(\mathbb R^d, 
\mathbb R)$ such that
\begin{align} \label{e2.51}
\int_{\mathbb R^d} (\Lambda^s f ) |f|^{p-2} f dx \le  -C_{p,d,s}<0,
\end{align}
where $C_{p,d,s}>0$ depends only on ($p$, $d$, $s$).

Furthermore for the same $p\in [p_0, \, \infty)$, there exists a sequence of Schwartz functions
$f_j:\; \mathbb R^d \to \mathbb R$ with frequency localized around $N_j \to \infty$, such that 
\begin{align*}
-D_2< \frac{\int_{\mathbb R^d } \Lambda^s  f_j |f_j|^{p-2} f_j dx } { N_j^s \int_{\mathbb R^d} 
|f_j|^p dx } <-D_1<0.
\end{align*}
In the above $D_1>0$, $D_2>0$ are constants depending  on  ($p$, $d$, $s$). More precisely, the frequency support of $f_j$
satisfies
\begin{align*}
\mathrm{supp}(\widehat{f_j} ) \subset \{ \xi:\,  \alpha_1 N_j < |\xi| < \alpha_2 N_j \},
\end{align*}
where $\alpha_1>0$, $\alpha_2>0$ are constants depending on ($p$, $d$, $s$).
\end{thm}
\begin{proof}
It suffices for us to prove \eqref{e2.51}. The frequency localized version follows from similar
arguments as in Theorem \ref{t1}. Fix $s>2$ and denote $K=e^{-\Lambda^s} \delta_0$ as
the kernel function corresponding to $e^{-\Lambda^s}$. By Theorem \ref{thm2.1}, we clearly
have $\| K \|_{L_x^1(\mathbb R^d)} >1$. By using the spatial decay of $K$, we have for
some $L_0$ sufficiently large
\begin{align}
\int_{\mathbb R^d} K(y) \mathrm{sgn}(K(y) ) \chi_{|y|\le L_0} dy >1.
\end{align}
By suitably mollifying the function $\mathrm{sgn}(K(y) ) \chi_{|y| \le L_0}$, we obtain
for some $\psi \in C_c^{\infty}(\mathbb R^d)$ with $\| \psi \|_{\infty}\le 1$ that
\begin{align}
\int_{\mathbb R^d} K(y) \psi (y) dy >1.
\end{align}
Thus for some $\beta_1>0$, 
\begin{align}
\| e^{-\Lambda^s} \psi \|_{\infty} \ge  (1+2\beta_1) \| \psi \|_{\infty}.
\end{align}
Since $\lim_{p \to \infty} \| \psi \|_p = \| \psi \|_{\infty}$ and
$\lim_{p\to \infty} \| e^{-\Lambda^s} \psi \|_p = \| e^{-\Lambda^s } \psi \|_{\infty}$, we
can find $p_0$ sufficiently large such that for all $p \in [p_0, \infty)$, 
\begin{align}
\| e^{-\Lambda^s} \psi \|_{p} \ge  (1+\beta_1) \| \psi \|_{p}.
\end{align}
Define $\psi_1 = \psi /\| \psi \|_p$. Clearly $\| \psi_1 \|_p=1$ and
\begin{align}
\int_0^1 \frac d {dt} \Bigl(  \| e^{-t \Lambda^s} \psi_1 \|_p^p \Bigr) dt
= \| e^{-\Lambda^s} \psi_1 \|_p^p  -1
\ge (1+\beta_1)^p -1 =:\tilde c_1>0.
\end{align}
Thus for some $t_0 \in (0,1)$, we must have
\begin{align}
\frac d {dt} \Bigl( \| e^{-t \Lambda^s } \psi_1 \|_p^p  \Bigr) \biggr|_{t=t_0} \ge \frac 12
\tilde c_1>0.
\end{align}
Denote $\psi_2 = e^{-t_0 \Lambda^s } \psi_1$. For some constant $\tilde c_2>0$,
we clearly have
\begin{align}
\int_{\mathbb R^d} \Lambda^s \psi_2 | \psi_2 |^{p-2} \psi_2 dx \ge \tilde c_2>0.
\end{align}
It follows that for some $\psi_3 \in C_c^{\infty}(\mathbb R^d)$ and some $\tilde c_3>0$,
\begin{align}
\int_{\mathbb R^d} \Lambda^s \psi_3 | \psi_3 |^{p-2} \psi_3 dx \ge \tilde c_3>0.
\end{align}
Thus \eqref{e2.51} is proved.
\end{proof}

\begin{thm} \label{thm3.3}
Let the dimension $d\ge 1$.  Let $s>2$.  There exists $1<p_1<2$ depending only on
($s$, $d$) such that for any $p \in (1, p_1]$, we can find $f \in C_c^{\infty}(\mathbb R^d, 
\mathbb R)$ such that
\begin{align} \label{e2.62}
\int_{\mathbb R^d} (\Lambda^s f ) |f|^{p-2} f dx \le  -C_{p,d,s}<0,
\end{align}
where $C_{p,d,s}>0$ depends only on ($p$, $d$, $s$).

Furthermore for the same $p\in (1, p_1]$, there exists a sequence of Schwartz functions
$f_j:\; \mathbb R^d \to \mathbb R$ with frequency localized around $N_j \to \infty$, such that 
\begin{align}
-D_2< \frac{\int_{\mathbb R^d } \Lambda^s  f_j |f_j|^{p-2} f_j dx } { N_j^s \int_{\mathbb R^d} 
|f_j|^p dx } <-D_1<0.
\end{align}
In the above $D_1>0$, $D_2>0$ are constants depending  on  ($p$, $d$, $s$). More precisely, the frequency support of $f_j$
satisfies
\begin{align}
\mathrm{supp}(\widehat{f_j} ) \subset \{ \xi:\,  \alpha_1 N_j < |\xi| < \alpha_2 N_j \},
\end{align}
where $\alpha_1>0$, $\alpha_2>0$ are constants depending on ($p$, $d$, $s$).
\end{thm}
\begin{proof}
We only need to show \eqref{e2.62}.  The idea is to use the construction in Theorem
\ref{thm2.2} and duality. Denote $T=e^{-\Lambda^s}$. 
Let $p_0>2$ be the same as in Theorem \ref{thm2.2}. Denote $p_1 = p_0/(p_0-1)$.
 For $p\in (1, p_1]$, denote $p^{\prime} = p/(p-1) \in [p_0, \infty)$. 
By using the proof of Theorem \ref{thm2.2}, we can find $f \in C_c^{\infty}(\mathbb R^d)$
with $\|f \|_{p^{\prime} } =1$ such that for some constant $\gamma_1>0$, 
\begin{align}
 \| T f \|_{p^{\prime} } \ge 1+2\gamma_1.
\end{align}
Since 
\begin{align}
\| Tf \|_{p^{\prime}} = \sup_{\|\psi \|_p =1} \langle \psi, T f \rangle, 
\end{align}
we can find $\psi \in C_c^{\infty} (\mathbb R^d)$ with $\| \psi \|_p =1$ such that
\begin{align}
1+\gamma_1 \le \langle \psi, T f \rangle = \langle T \psi, f \rangle \le \|T \psi \|_p 
\|f \|_{p^{\prime} } = \| T \psi \|_p.
\end{align}
We can then use the inequality
\begin{align}
\int_0^1 \frac d {dt} \Bigl( 
\| e^{-t\Lambda^s} \psi \|_p^p \Bigr) dt =\| e^{-\Lambda^s} \psi \|_p^p -1
\ge (1+\gamma_1)^p -1
\end{align}
to obtain \eqref{e2.62}.
\end{proof}

\begin{rem}
The previous theorems show the failure also for $s>2$ of the Strook-Varopoulos inequality.
\end{rem}

\begin{rem}
In \cite{Lieb89}, Lieb considered maximizers for the problem:
\begin{align}
\sup_f \frac {\| \mathcal G f \|_q} { \|f \|_p},
\end{align}
where $\mathcal G$ is an integral operator with Gaussian kernel $G$, and $1<p, q<\infty$.  
For degenerate and centered Gaussian kernel $G$ (see equation (1.3) in \cite{Lieb89}) the supremum can be shown to be taken over centered Gaussian functions. In particular
if we consider the problem\footnote{Note that the kernel corresponding to
$e^{\Delta}$ is $K(x,y)=(4\pi)^{-d} e^{-\frac{|x-y|^2} 4}$ which is degenerate in the
language of \cite{Lieb89}.}
\begin{align}
\sup_{f} \frac {\| e^{\Delta} f \|_p} { \|f \|_p},
\end{align}
for $p\in (1,\infty)$, then it is clear that one may take $f_n = e^{t_n\Delta} \delta_0$, with
$t_n\to \infty$ as $n\to \infty$ in order to saturate the optimal operator norm bound $1$. 
On the other hand, for general signed kernel $G$ an intriguing problem is to classify the
maximizers or the maximizing sequence. These type of results will improve our understanding
of the Bernstein-type inequalities. 
\end{rem}

\section{Bernstein inequality for the periodic case}
In this section we show some positive results for the fractional Laplacian operator
$\Lambda^s$, $0<s \le 2$ on the periodic torus. 
Let $\mathbb T^d= \mathbb R^d/\mathbb Z^d=[-\frac 12, \frac 12]^d$. 

For any integrable $f: \mathbb T^d \to \mathbb C$, denote
\begin{align*}
\boxed{
 \langle f\rangle  = \int_{\mathbb T^d} f(x) dx.}
\end{align*}
We use the following convention for Fourier transform on $\mathbb T^d$: 
\begin{align}
& \widehat f (k) = \int_{\mathbb T^d} f(x) e^{-2\pi i k \cdot x} dx, \qquad k \in \mathbb Z^d; \\
& f(x) = \sum_{k \in \mathbb Z^d} \widehat f (k) e^{2\pi i k \cdot x}, \qquad x \in \mathbb T^d.
\end{align}
The fractional laplacian operator $\Lambda^s =(-\Delta)^{s/2}$, $s>0$ on $ \mathbb T^d$ is defined as
\begin{align}
\widehat{\Lambda^s f } (k) = |k|^s \widehat f (k), \qquad k \in \mathbb Z^d.
\end{align}
In yet other words it corresponds to the Fourier multiplier $|k|^s$.  Note that $\widehat f (0) = \langle f \rangle$.

\begin{thm}[Bernstein inequality on the torus] \label{thm3.1}
Let $0<s\le 2$ and consider $\Lambda^s$ on $\mathbb T^d= [-\frac 12, \frac 12]^d$, $d\ge 1$. 
Let $1<p<\infty$. 
For any smooth $f:\mathbb T^d \to \mathbb R$ with $\langle f \rangle=0$, we have
\begin{align} \label{e4.34}
\| e^{  - t \Lambda^s } f \|_p \le e^{-c_{p,s,d} t } \| f \|_p,  \qquad \forall\, t>0.
\end{align}
Here $c_{p,s,d}>0$ depends only on ($p$, $s$, $d$).  Consequently for any
smooth $f$ with $\langle f \rangle =0$ we have
\begin{align} \label{4.39a}
\int_{\mathbb T^d} (\Lambda^s f ) |f|^{p-2} {f} dx \ge \tilde c_{p,s,d} \| f \|_p^p,
\end{align}
where $\tilde c_{p,s,d}>0$ depends only on ($p$, $s$, $d$). 
\end{thm} 
\begin{rem}
Similar results hold if $f$ is complex-valued or vector-valued. For example if $f:\mathbb T^d
\to \mathbb R^{d_1}$ and $\int_{\mathbb T^d} f dx =0$, then we have
\begin{align}
\| | e^{-t \Lambda^s f } | \|_p \le  e^{-c t } \| |f | \|_p,
\end{align}
where $|f| =\sqrt{f_1^2 + \cdots + f_{d_1}^2}$. For complex-valued $f$,  \eqref{4.39a}
should be replaced by
\begin{align}
\int_{\mathbb T^d} (\Lambda^s f ) |f|^{p-2} {f^*} dx \ge \tilde c_{p,s,d} \| f \|_p^p,
\end{align}
where $f^*$ denotes the complex conjugate of $f$. 
\end{rem}

\begin{rem}
We briefly explain the heuristics as follows. Consider the case $s=2$, i.e. the usual
Laplacian $\Delta = -\Lambda^2$.  Clearly we have
\begin{align*}
&\| e^{t \Delta} f \|_2 \le e^{-ct} \| f\|_2, \qquad \forall\, f \text{ with $\langle f \rangle=0$};\\
&\| e^{t\Delta} f \|_{\infty} \le \| f \|_{\infty}; \\
& \| e^{t\Delta} f \|_1 \le \| f \|_1.
\end{align*}
By \emph{formally} interpolating the above two inequalities,  it is natural to expect that
for any $p\in (1, \infty)$, 
\begin{align*}
\| e^{t \Delta} f \|_p \le e^{-c_p t} \| f \|_p, \qquad \forall\, f \text{ with $\langle f \rangle=0$},
\end{align*}
where $c_p>0$.  However due to the presence of the constraint $\langle f \rangle=0$, this requires some nontrivial interpolation of Riesz-Thorin type.  The technical difficulty is that the usual Riesz-Thorin interpolation employs
a nonlinear functor which in general does not preserve the condition $\langle f \rangle =0$. 
Nevertheless in
 Theorem \ref{thm3.1} we overcome this difficulty by proving some
nontrivial concentration-type inequalities. 
\end{rem}

\begin{lem}[Strong Phragman-Lindelof estimate]
Suppose $h$ is an analytic function on the strip $0<\mathrm{Re}(z)<1$ and is continuous
up to the boundary. Assume for some constant $\alpha<\pi$ and constant $A$, 
\begin{align}
|h(z) | \le e^{A e^{a |\mathrm{Re}(z) | } }, \qquad \forall\, \text{$z$ in the closed strip}.
\end{align}
Then for any $0<\theta<1$, we have
\begin{align} \label{e3.75}
|h(\theta)| \le \exp\left( \frac {\sin \pi \theta} 2 
\int_{-\infty}^{\infty} \Bigl(
\frac {\log |h(iy) |} {\cosh \pi y - \cos \pi \theta} 
+ \frac { \log |h(1+iy)|} { \cosh \pi y + \cos \pi \theta} \Bigr) dy  \right).
\end{align}
\end{lem}
\begin{proof}
See for example Chapter 5.4 of \cite{Stein}.
\end{proof}

\begin{lem}[Small mean implies short-time decay] \label{lem_Small}
Let $s>0$ and consider the torus $\mathbb T^d$, $d\ge 1$. Suppose $f: \; \mathbb T^d
\to \mathbb C$ and  $f \in L^2$. If $\frac 1 {\|f\|_2} |\int_{\mathbb T^d} f dx | \le \lambda <1$, then
\begin{align}
\| e^{-t \Lambda^s } f \|_2 \le e^{-\alpha_1 t } \|f \|_2, \quad \forall\, 0<t<t_0.
\end{align}
Here $\alpha_1>0$, $t_0>0$  are constants depending only on ($s$, $d$, $\lambda$). 
\end{lem}
\begin{proof}
With no loss we assume $\|f \|_2 =1$. 
Denote $\langle f \rangle = \int_{\mathbb T^d} f dx$. 
Clearly 
\begin{align}
\| e^{-t\Lambda^s} f \|_2^2 & \le e^{-ct} \| f- \langle f \rangle \|_2^2 + | \langle f \rangle |^2 \\
& \le e^{-ct} ( \|f \|_2^2 - |\langle f \rangle|^2) + |\langle f \rangle|^2 \\
& \le e^{-ct} + (1-e^{-ct} ) \lambda^2 \le e^{-c_1t}, 
\qquad \text{for $0<t\ll 1$}. 
\end{align}

\end{proof}

\begin{proof}[Proof of Theorem \ref{thm3.1}]
We shall present the proof for the simplest case $s=2$, $d=1$ and $2<p<\infty$. 
It is not difficult to adapt the proof to the most general situations. 

It suffices for us to prove \eqref{e4.34} for $0<t\le t_0$ where $t_0>0$ can be taken as a small
constant depending on ($s$, $d$, $p$).  This is because for $t>t_0$,
\begin{align}
\| e^{-t \Lambda^s } f \|_p & = \| e^{-t_0 \Lambda^s} e^{-(t-t_0)\Lambda^s} f \|_p  \\
& \le e^{-c t_0} \| e^{-(t-t_0) \Lambda^s  }f\|_p, \qquad (\text{since $\langle e^{-(t-t_0)\Lambda^s} f \rangle =0$}).
\end{align}
One can then iterate the estimates to get the decay for all $t>0$. 

Let $\frac 1 {p} = \frac {1-\theta}2$, $0<\theta<1$. Take simple real-valued functions $f$, $g$ with $\langle f \rangle=0$,
$\| f \|_{p}=1$, $\| g \|_{p^{\prime} }=1$ (here $ p^{\prime}=p/(p-1)$).  Consider
\begin{align*}
h(s) = \langle e^{t\Delta} ( |f|^{p (\frac {1-s}2 + \frac{s}{\infty})} \operatorname{sgn}(f)),
|g|^{p^{\prime} (\frac{1-s}{2} + \frac{s}1) }
\operatorname{sgn}(g) \rangle,
\end{align*}
where $\langle f_1,f_2 \rangle :=\int_{\mathbb T} \bar f_1 f_2 dx$. Here $\bar f_1$ denotes
the complex conjugate of $f_1$. 

(Here we recall the usual Riesz-Thorin setup: namely in going from
$L^{p_0} \to  L^{q_0}$, $L^{p_1} \to L^{q_1}$ to $L^p \to L^{q}$, one needs to employ the 
general interpolation formula for simple functions $f$ and $g$:
\begin{align*}
& f_z = |f|^{p (\frac {1-z}{p_0} + \frac {z} {p_1} )} \mathrm{sgn}(f), \\
& g_z= |g|^{q (\frac{1-z}{q_0^{\prime} } + \frac z {q_1^{\prime} } )} \mathrm{sgn}(g),
\end{align*}
where $q_j^{\prime}$ are conjugates of $q_j$. Our case corresponds to $p_0=q_0=2$, $p_1=q_1=\infty$.)

We verify the interpolation as follows. 

\begin{itemize}
\item The case $\operatorname{Re}(s)=1$. Clearly
\begin{align}
|h(s)| \le \| e^{t\Delta} ( |f|^{p \frac {-i \mathrm{Im}(s) }2} \mathrm{sgn}(f) ) \|_{\infty}
\| |g|^{p^{\prime}}  \|_1 \le \| g\|_{p^{\prime}}^{p^{\prime}} =1.
\end{align}

\item  The case $\operatorname{Re}(s)=0$, i.e. $s=iy$, $y\in \mathbb R$. 
First for all $y \in \mathbb R$, we clearly have
\begin{align}
|h(iy) | & \le \| e^{t\Delta} ( |f|^{p \frac {1-iy}2} \mathrm{sgn}(f) ) \|_2
\| |g|^{p^{\prime}(\frac {1-iy}2+iy)} \mathrm{sgn}(y) \|_2 \\
& \le \| |f|^{\frac p 2} \|_2 \| |g|^{\frac 12 p^{\prime} } \|_2 \le 1.
\end{align}
It remains for us to show that  for $|y| \le  1$ and $0<t\le t_0$ (for some small $t_0>0$),
\begin{align}
|h(iy)|\le e^{-ct}, \label{30}
\end{align}
where $c>0$ is some constant. 
If this holds, we can just use Strong Phragman-Lindelof Theorem to conclude the interpolation argument. Indeed by using \eqref{e3.75}, we have
\begin{align}
|h(\theta)| \le \exp\left( \frac {\sin \pi \theta}2 \int_{|y|\le 1} \frac {-ct} {\cosh \pi y -
\cos\theta} dy \right) \le e^{-\tilde c t},
\end{align}
where $\tilde c>0$ is a constant. 

\item It remains for us  to verify \eqref{30}. By Lemma \ref{lem_Small}, it suffices for 
us to establish for $|y|\le 1$,
\begin{align}
| \mathbb E (  (|f|^p)^{\frac{1-iy}2} \operatorname{sgn}(f) ) | \le \lambda<1, \label{40}
\end{align}
where $\lambda>0$ is some constant.  Here and below we denote 
\begin{align}
\mathbb E h = \int_{\mathbb T} h dx.
\end{align}
We recall that $\|f\|_p=1$ and $\langle f \rangle =0 =\mathbb E f$. 

The proof of \eqref{40} follows from the following steps. 

\item If $\mathbb E|f| \le \lambda_1 <1$, then by using the interpolation
$\|f \|_{\frac p2} \le \|f \|_p^{\frac {p-2}{p-1}} \| f \|_1^{\frac 1 {p-1} }$, we have
\begin{align*}
| \mathbb E ( |f|^{\frac p2} ) | \le \lambda_1^{ \frac 1 {p-1} \cdot \frac p2} <1.
\end{align*}
This clearly implies \eqref{40}. Therefore we can assume $\lambda_1 < \mathbb E |f| \le 1$ and 
$\lambda_1 \to 1-$. Since $\|f \|_2 \le \|f \|_p \le 1$, we have 
\begin{align*}
\mathbb E | |f|-1|^2  \le 2- 2 \mathbb E |f| \le 2(1-\lambda_1)=:\delta_1 \ll 1.
\end{align*}
We shall view $\delta_1$ as a tunable parameter which can be taken sufficiently small.

\item By using the inequality $|x^{\frac p2} -1 | \lesssim |x-1| \langle x \rangle^{\frac {p-2}2}$, we have
\begin{align*}
\mathbb E ||f|^{\frac p2} -1| \le \delta_2 = O( \delta_1^{\frac 12} )\ll 1.
\end{align*}

\item  We now take $\eta>0$ whose smallness will be specified momentarily. Clearly
\begin{align}
| \mathbb E (  (|f|^p)^{\frac{1-iy}2} \operatorname{sgn}(f) ) |
& \le | \mathbb E (  (|f|^p)^{\frac{1-iy}2} \operatorname{sgn}(f) ) \chi_{|f|\ge \eta} |  +
\eta^{\frac p2} \\
& \le \mathbb E ||f|^{\frac p2} -1| +
\mathbb E | e^{-i\frac p2 y \log |f|} -1| \chi_{|f|\ge \eta} + |\mathbb E \mathrm{sgn}(f) | +
\eta^{\frac p2} \\
& \le \mathbb E ||f|^{\frac p2} -1|
+ \frac p2 \mathbb E |\log |f| | \chi_{|f|\ge \eta} + |\mathbb E \mathrm{sgn} (f) | +
\eta^{\frac p2}.
\end{align}

\item Observe that for $x\ge \eta$ ($\eta<1$ will be taken sufficiently small), we have
\begin{align}
|\log x|=|\log x - \log 1| \le \frac 1 {\eta} |x-1|.
\end{align}
Thus
\begin{align}
\mathbb E |\log |f| | \chi_{|f|\ge \eta} \le \frac 1 {\eta} \mathbb E ||f|-1|.
\end{align}
On the other hand,  observe (below we use the crucial property that $\mathbb E f =\mathbb E |f| \mathrm{sgn}(f) =0$) 
\begin{align*}
|\mathbb E \mathrm{sgn}(f) |= |\mathbb E  \operatorname{sgn} f (|f|-1) | \le \mathbb E ||f|-1|.
\end{align*}

\item Thus we obtain
\begin{align}
| \mathbb E (  (|f|^p)^{\frac{1-iy}2} \operatorname{sgn}(f) ) |
&\le  \mathbb E | |f|^{\frac p2} -1| +  ( \frac p{2\eta}+1) \mathbb E ||f|-1| +\eta^{\frac p2} \\
&\le  O(\delta_1^{\frac 12})  \cdot (1+ \frac p{2\eta} ) + \eta^{\frac p2}.
\end{align}
Taking $\eta=\delta_1^{\frac 14}$ with $\delta_1$ sufficiently small clearly yields the result.
\end{itemize}
\end{proof}

In what follows, we shall explain a somewhat more simplified approach to the proof of Theorem
\ref{thm3.1}. 

We begin with a simple yet powerful lemma.
\begin{lem} \label{lem3.3}
Let $\Omega=\mathbb R^d$ or the periodic torus $\mathbb T^d$.
Suppose $K\in L^1(\Omega)$ is nonnegative with unit $L^1$ mass. For any
$p \in [2, \infty)$, we have
\begin{align}
\| K * f \|_{L^p(\Omega)} \le \| K * ( |f|^{\frac p2} ) \|_{L^2(\Omega)}^{\frac 2p}.
\end{align}
Here $*$ denotes the usual convolution, i.e. 
\begin{align}
(K*f)(x) = \int K(x-y) f(y) dy.
\end{align}
For $p \in (1,2]$, we have
\begin{align}
\| K*f \|_p \le \| K*( |f|^{\frac p2} ) \|_2^2 \cdot \| f \|_p^{2-p}.
\end{align}
\end{lem}
\begin{proof}
Observe that for each fixed $x$, $K(x-y)dy$ can be viewed as a probability measure. Thus
if $p\in [2, \infty)$,  then
\begin{align}
\int |f(y) | K(x-y) dy \le \left( \int |f(y)|^{\frac p2} K(x-y) dy \right)^{\frac 2p}.
\end{align}
This yields  the first inequality. Now for $p\in (1,2)$, 
by using the inequality  $\| g \|_{\frac 2p} \le \|g\|_1^{p-1} \| g \|_2^{2-p}$
with $g= |f|^{\frac 2p}$ and $d\mu = K(x-y) dy$,  
we have
\begin{align}
\int |f (y) | K(x-y) dy 
\le \left( \int |f(y)|^{\frac p2} K(x-y) dy \right)^{\frac {2(p-1)} p}
\left( \int |f(y)|^{p}  K(x-y) dy \right)^{\frac {2-p} p}.
\end{align}
Thus
\begin{align}
\| K* f \|_p \le \| K*( |f|^{\frac p2} ) \|_2^2 \cdot \| f \|_p^{2-p}.
\end{align}
\end{proof}

We now sketch a different proof of Theorem \ref{thm3.1} for the Laplacian case (i.e.
$-\Lambda^2=\Delta$) as follows. With no loss we consider the case
$\mathbb T^d=\mathbb T$ and $p \in (2,\infty)$.  Take $f$ with mean zero
and $\|f\|_p=1$. Discuss two cases. 
\begin{itemize}

\item 
\underline{Case 1}: $\|f\|_{\frac p2+1} \ll 1$.  Clearly then $\|f \|_{\frac p 2} \ll 1$.  By Lemma
\ref{lem3.3}, we obtain
 \begin{align}
 \| e^{t\Delta } f \|_p \le \| e^{t\Delta} (|f|^{\frac p2} ) \|_2^{\frac 2p}.
\end{align}
By Lemma \ref{lem_Small}, since $\mathbb E |f|^{\frac p2} \ll 1$ and $\| |f|^{\frac p2} \|_2 =1$,
we obtain
\begin{align}
\| e^{t\Delta} ( |f|^{\frac p2} ) \|_2 \le e^{-ct} \|  |f|^{\frac p2} \|_2 =e^{-ct},
\qquad 0<t\le t_0.
\end{align}
This clearly implies the desired estimate $\|e^{t\Delta } f \|_p \le e^{-\tilde c t} \|f \|_p$ for
$0<t\le t_0$. Note that this part of the argument can be adapted to $e^{-t\Lambda^s}$ for
$0<s<2$.

\item
\underline{Case 2}: $\|f\|_{\frac p2+1} \gtrsim 1$.  Note that
\begin{align*}
\langle |f|^{\frac p2} \operatorname{sgn}(f), f \rangle = \int |f|^{\frac p2+1} dx \gtrsim 1.
\end{align*}
Since $f$ is spectrally localized to $|k|\ge 1$,  it follows that (below $P_k$ is the Fourier
projection to all modes $|k|\ge 1$)
\begin{align*}
\| P_{|k|\ge 1} ( |f|^{\frac p2} \operatorname{sgn}(f) ) \|_2 \| f \|_2 \gtrsim 1.
\end{align*}
On the torus, we obviously have $\|f \|_2 \le \|f\|_p =1$.  Thus
\begin{align}
\| P_{|k|\ge 1} ( |f|^{\frac p2} \mathrm{sgn}(f) ) \|_2 \gtrsim 1.
\end{align}
In yet other words,  the $L^2$-mass of 
$|f|^{\frac p 2} \operatorname{sgn}(f)$ must have a nontrivial portion in $|k| \ge 1$.  Now
observe that
\begin{align}
\int (-\Delta f ) |f|^{p-2} f dx & = \mathrm{const} \int | \nabla f |^2 |f|^{p-2} dx \\
& = \mathrm{const} \int | \nabla ( |f|^{\frac p2} \mathrm{sgn} (f ) ) |^2 dx \\
&\gtrsim \| P_{|k|\ge 1} ( |f|^{\frac p2} \mathrm{sgn}(f) ) \|_2^2
\gtrsim 1 = \|f\|_p^p.
\end{align}
Thus the desired inequality follows.
\end{itemize}

Next we shall state and prove a frequency localized Bernstein inequality on the torus. 
Let $\psi \in C_c^{\infty}(\mathbb R^d)$ be such that $\psi(\xi)=1$ for $|\xi|\le 1$
and $\psi(\xi) =0$ for $|\xi|\ge 1.01$. For integer $N\ge 2$ and $f:\, \mathbb T^d \to \mathbb C$,
define 
\begin{align}
\widehat{P_N f}(k) = \widehat f (k) ( \psi(\frac {k}{2N} ) -  \psi ( \frac k N) ).
\end{align}
In yet other words, $P_N$ is a smooth frequency projection to $\{ |k| \sim N \}$.  Here on the
torus we use the convention
\begin{align}
&\widehat{f}(k) = \int_{\mathbb T^d} f(x) e^{-i 2\pi x \cdot k } dx; \\
& f(x) = \sum_{k \in \mathbb Z^d} \widehat f(k) e^{2\pi i k \cdot x}.
\end{align}

We need the following lemma from Kato \cite{Kato}. The inequality stated therein\footnote{
Note that there is a minor typo in the definition of $Q_p$ in formula (2.2), pp 55 of \cite{Kato}: the lower limit for the
 integration therein should be $\phi(x) \ne 0$ instead of $\partial \phi(x) \ne 0$.}
is for the whole space. We adapt it here for the torus with essentially the same proof.
\begin{lem}[Kato \cite{Kato}] \label{lem_Kato}
Let $1<p<\infty$. Assume $\phi \in W^{2,p}(\mathbb T^d \to \mathbb R^{d_1} )$, where
$d\ge 1$, $d_1\ge 1$. Then
\begin{align}
- \langle |\phi|^{p-2} \phi, \, \Delta \phi \rangle
\ge  \min\{1,  p-1\}  \int_{\phi \ne 0} |\nabla \phi|^2 |\phi|^{p-2} dx.
\end{align}
Here for $f$, $g:\mathbb T^d \to \mathbb R^{d_1}$, 
\begin{align}
\langle f,\, g \rangle = \sum_{j=1}^{d_1} \int_{\mathbb T^d} f_j g_j dx.
\end{align}
\end{lem}
\begin{proof}
We briefly recall Kato's proof. For $p>2$,  we use the identity
\begin{align}
 - \langle |\phi|^{p-2} \phi, \Delta \phi \rangle
& = \int |\phi|^{p-2} |\partial_k \phi_j|^2  + (p-2) |\phi|^{p-2} \frac{\phi_l \phi_j}{|\phi|^2} \partial_k \phi_l
 \partial_k \phi_j  \\
 &\ge  \int |\phi|^{p-2} |\partial_k \phi_j|^2.
 \end{align}
 For $1<p<2$, we use
 \begin{align}
 - \langle |\phi|^{p-2} \phi, \Delta \phi \rangle
= -\lim_{\epsilon \to 0} 
\langle ( |\phi|^2 + \epsilon)^{\frac {p-2}2} \phi, \, \Delta \phi \rangle.
\end{align}
Denote $\phi_{\epsilon} = \sqrt{|\phi|^2+\epsilon}$. Then
(note below $p-2<0$ and $\frac {\phi_l \phi_l} {|\phi_{\epsilon} |^2}$ is bounded by $1$ in matrix
norm)
\begin{align}
 - \langle |\phi_{\epsilon}|^{p-2} \phi, \Delta \phi \rangle
& = \int |\phi_{\epsilon} |^{p-2} |\partial_k \phi_j|^2  + (p-2) |\phi_{\epsilon} |^{p-2} \frac{\phi_l \phi_j}{|\phi_{\epsilon}|^2} \partial_k \phi_l
 \partial_k \phi_j  \\
 &\ge  (p-1) \int |\phi_{\epsilon} |^{p-2} |\partial_k \phi_j|^2
 \ge (p-1) \int_{\phi\ne 0} |\phi_{\epsilon} |^{p-2} | \nabla \phi|^2.
 \end{align}
The result follows from dominated convergence (for the LHS) and monotone convergence
(for the RHS).

\end{proof}

\begin{thm}[Bernstein inequality on the torus, frequency localized version]  \label{thm4.2a}
Let $0<s\le 2$ and consider $\Lambda^s$ on $\mathbb T^d= [-\frac 12, \frac 12]^d$, $d\ge 1$. 
Let $1<p<\infty$. 
For any smooth $f:\mathbb T^d \to \mathbb R$ and any integer $N\ge 2$, we have
\begin{align} \label{e4.83}
\| e^{  - t \Lambda^s } P_N f \|_p \le e^{-c_{p,s,d} N^s t } \|  P_N f \|_p,  \qquad \forall\, t>0.
\end{align}
Here $c_{p,s,d}>0$ depends only on ($p$, $s$, $d$, $\psi$) (Recall $\psi$ is the same cut-off
function used in the
definition of the operator $P_N$).  Consequently 
\begin{align} \label{e4.84}
\int_{\mathbb T^d} (\Lambda^s P_Nf ) |P_N f|^{p-2} P_N f dx \ge \tilde c_{p,s,d}  N^s \|P_N f \|_p^p,
\end{align}
where $\tilde c_{p,s,d}>0$ depends only on ($p$, $s$, $d$, $\psi$).
\end{thm} 
\begin{rem}
See \cite{L12} for a proof using a nontrivial perturbation of the L\'evy semigroup near low frequencies.
\end{rem}

\begin{proof}
We follow \cite{CL12}. For the Laplacian case, the idea  is based on an ingenious partial integration trick dating back 
to Danchin \cite{D97} ($p$ being even integers), Planchon \cite{P00}  ($p>2$) and Danchin \cite{D02} ($1<p<2$). 

To simplify the notation we shall write $f= P_N f$ keeping in mind that $f$ is frequency-localized. 
We shall write $\int_{\mathbb T^d} dx$ simply as $\int$. 

Step 1. Laplacian case.  We first show
\begin{align}
- \int \Delta  f |f|^{p-2} f  \gtrsim N^2 \|f \|_p^p.
\end{align}
We first deal with the case $p>2$. We have
\begin{align}
\|f \|_p^p &= \int f^2 |f|^{p-2} = \int (\nabla \cdot \nabla \Delta^{-1} f ) f |f|^{p-2}  \\
& \lesssim \int |\nabla \Delta^{-1} f | | \nabla f | |f|^{p-2} \\
& \le C_\epsilon N^{-2} \int |\nabla f|^2 |f|^{p-2} + {\epsilon} N^2 \int |\nabla \Delta^{-1} f |^2 |f|^{p-2} \\
& \le C_\epsilon N^{-2} \int |\nabla f |^2 |f|^{p-2}
+{\epsilon} \cdot \mathrm{Const} \|  f\|_p^{p}.
\end{align}
Choosing $\epsilon$ to be sufficiently small then yields the result for $p>2$. 
For $1<p<2$,  we use
\begin{align}
N^p \| f \|_p^p & \lesssim \| \nabla f \|_p^p = \int_{f \ne 0} ( |\nabla f |^2 |f|^{p-2} )^{\frac p2}
(|f|^p)^{\frac {2-p} 2} dx \\
& \lesssim  (\int_{f\ne 0} |\nabla f|^2 |f|^{p-2} )^{\frac p2} (\int |f|^p )^{\frac {2-p}2}.
\end{align}
The desired result then follows from Lemma \ref{lem_Kato}.

Step 2: The estimate \eqref{e4.83} in the case $s=2$ follows from Step 1 by examining 
$\frac d{dt} \Bigl( \|e^{-t \Lambda^s }f \|_p^p \Bigr)$ and an energy estimate.  The general case
$0<s<2$ follows from subordination.  The estimate \eqref{e4.84} follows from
differentiating at $t=0$. 
\end{proof}

\section{Liouville theorem for general fractional Laplacian operators}
We now consider the fractional heat equation of the form 
\begin{align} \label{ancient_eq_00}
\partial_t u = - \Lambda^s u, \quad (t,x) \in (-\infty, 0] \times \mathbb R^d.
\end{align}
Here $\Lambda^s = (-\Delta)^{s/2}$ is the fractional Laplacian of order $s$, and we assume
$s>0$. Note that for $0<s\le 2$ the corresponding semigroup has positivity but this is no longer
the case for $s>2$, i.e. the higher order Laplacians. 
% Here one should recall a useful fact 
%that reads as follows: $\phi(\xi)= \Phi(|\xi|^2)$ is the Fourier transform of a probability
%measure on $\mathbb R^d$, if and only if that $\Phi$ is a completely monotone function on 
%$(0,\infty)$.  

\begin{thm} \label{thm5.1}
Suppose $u$ is an ancient solution  to \eqref{ancient_eq_00} satisfying
\begin{align*}
|u(t,x)| \le \frac {C} {|x|^a}, \quad \forall\, (t,x) \in (-\infty,0] \times \mathbb R^d,
\end{align*}
where $0<a<d$ and $C>0$ are constants. Then $u$ must be identically zero. 
\end{thm}

\begin{proof}
Take any $\phi \in C_c^{\infty}(\mathbb R^d)$ and consider
$u_{\phi} = \phi* u$.  By splitting into $|y|\le 1$ and $|y|>1$ respectively, it is not difficult to check that $u_{\phi}$ is smooth and
$u_{\phi} \in L^p$ for any $\frac d{a}<p<\infty$. Fix any $t_0 \in (-\infty, 0]$. We then
have $u_{\phi}(t_0)=e^{-(t_0-t) \Lambda^s } u_{\phi} (t)$ for any $t<t_0$. By sending $t$
to $-\infty$ and invoking the usual decay estimates (for the kernel $e^{-\tau \Lambda^s}$), i.e.
\begin{align*}
\| u_{\phi}(t_0)\|_{\infty} \lesssim (t_0-t)^{-\frac ds\cdot \frac 1 {p-1}} \| u_{\phi}(t)\|_p  \lesssim 
(t_0-t)^{-\frac ds\cdot \frac 1 {p-1}},
\end{align*}
we obtain
$\| u_{\phi}(t_0)\|_{\infty}=0$. Thus $u_{\phi}(t_0)=0$ for any $\phi \in C_c^{\infty}$. This
 implies that $u$ must be identically zero.
\end{proof}
\begin{rem}
The hypothesis that $|u|\lesssim |x|^{-a}$ can be replaced by the more general condition 
that $$\sup_{t} \|u(t) \|_{ L^{p_1} + L^{p_2}}<\infty$$ for some $1\le p_1,p_2<\infty$. 
\end{rem}
\begin{rem}
The rigidity result Theorem \ref{thm5.1} plays an important role in e.g. the infinite-time
blow up problem for the half-harmonic map flow from $\mathbb R$ to $\mathbb S^1$
\cite{SWZ21}.
\end{rem}

\appendix
\section{Computation of the contour integral}
In this appendix we show \eqref{2.11A}. Recall that 
$f(x) =\log (1+x^2)$, $f^{\prime}(x) = \frac {2x}{1+x^2}$, $f^{\prime\prime}(x)
=2\frac {1-x^2}{(1+x^2)^2}$ and we need to show
\begin{align} \label{A.102}
3 \int f^2 (f^{\prime\prime})^2 =12 \int_{-\infty}^{\infty}
(\log (1+x^2) )^2  \frac {(1-x^2)^2} {(1+x^2)^4} dx =
-\frac {29}6 \pi + \pi^3 + (\log 4) (-7+ \log 64) \pi.
\end{align}

%To this end, we denote 
%\begin{align}
%g(z) = 12 \frac { (1-z^2)^2} {(1+z^2)^4}.
%\end{align}
%Note that $g$ has poles at $z= \pm i$. 
For this we need to compute
\begin{align}
I_j= \int_{-\infty}^{\infty} (\log (1+x^2) )^2 \frac 1 {(1+x^2)^j} dx, \quad j=1,\cdots,4.
\end{align}

We shall proceed in several steps. 

Step 1. Preliminary reduction. Observe that
\begin{align}
I_{n+1} & = \int_{-\infty}^{\infty} (\log (1+x^2) )^2
\frac { 1+x^2 - x^2} {(1+x^2)^{n+1} } dx \\
& = I_n + \int_{-\infty}^{\infty} ( \log (1+x^2) )^2
\cdot \frac x {2n} \frac d {dx}  \Bigl(  (1+x^2)^{-n} \Bigr) dx \\
& =(1-\frac 1 {2n} ) I_n - \frac 2 n \int_{-\infty}^{\infty}
\frac {x^2} {(1+x^2)^{n+1} } \log(1+x^2) dx \\
& = (1-\frac 1{2n} )I_n 
- \frac 2 n \Bigl(\int_{-\infty}^{\infty} (1+x^2)^{-n} \log(1+x^2) dx
-\int_{-\infty}^{\infty} (1+x^2)^{-n-1} \log (1+x^2) dx \Bigr).
\end{align}

By using the above iterative relation, to compute $I_n$ for all $n$, it suffices for us to compute
\begin{align}
I_1 = \int_{-\infty}^{\infty} (\log (1+x^2) )^2 \frac 1 {1+x^2} dx
\end{align}
and 
\begin{align}
F_n = \int_{-\infty}^{\infty} (1+x^2)^{-n} \log (1+x^2) dx, \qquad n=1,\cdots, 4.
\end{align}

Step 2. Computation of $I_1$.  We shall perform a contour integral
computation.

For $z=\rho e^{i\theta}$ with $-\pi \le \theta <\pi$, we denote
\begin{align}
\mathrm{Log} z = \log \rho + i \theta.
\end{align}
In yet other words we use the standard principal branch of the multi-valued function $\log z$ with
argument in $[-\pi, \pi)$.  By a slight abuse of notation, we shall write $\mathrm{Log} z$
simply as $\log z$. 

Denote $g(z) = (1+z^2)^{-1}$. Note that $g$ has poles at $z=\pm i$.
First we observe that
\begin{figure}[!h]
\includegraphics[width=0.4\textwidth]{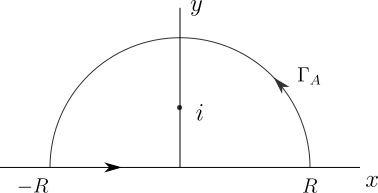}
\caption{Contour $\Gamma_A$}\label{fig:5}
\end{figure}

\begin{align}
&\int_{-\infty}^{\infty}
( \log (x+i) )^2  g(x) dx = 
\lim_{R\to \infty} \int_{\Gamma_A} (\log (z+i) )^2 g(z) dz =
2\pi i  \mathrm{Res} ( (\log (z+i) )^2 g(z); i ).  \label{100a}
%\\
%&\int_{-\infty}^{\infty} (\log(x-i ) )^2 g(x) dx
%= \lim_{R\to \infty}
%\int_{\Gamma_B} (\log z)^2 g(z) dz
%= 2\pi i \mathrm{Res} ( ( \log z)^2 g(z); -i). \label{100b}
\end{align}

By using our choice of the branch cut for
the logarithm function, we have for $x>0$
\begin{align}
& \log (x+i) = \frac 12 \log (x^2+1) + i\theta_x, \qquad \theta_x= \frac {\pi}2-
\arctan x; \\
& \log(-x+i)= \frac 12 \log (x^2+1) +i (\pi-\theta_x).
\end{align}

Thus \eqref{100a} becomes
\begin{align}
\int_0^{\infty}  \frac {\frac 12 (\log(1+x^2) )^2} {1+x^2} dx
-\int_0^{\infty} \frac { \theta_x^2 +(\pi-\theta_x)^2} {1+x^2} dx
=\mathrm{Re} \Bigl( 2\pi i  \mathrm{Res} ( (\log (z+i) )^2 g(z); i ) \Bigr).
\end{align}
This implies
\begin{align}
{I_1=2\int_0^{\infty} \frac {(\log(1+x^2) )^2} {1+x^2} dx
=4 \frac {\pi^3} 3 + 4 \mathrm{Re} \Bigl( 2\pi i  \mathrm{Res} ( (\log (z+i) )^2 g(z); i ) \Bigr).}
\end{align}
Since $\mathrm{Res}((\log (z+i) )^2 g(z); i)=\frac 18 i (\pi -2i \log 2)^2$, we obtain
\begin{align}
\boxed{I_1=\frac 13 \pi^3+4\pi (\log 2)^2.
}
\end{align}

Step 3. Computation of $F_n$.  This is analogous to the previous step. Note that
\begin{align}
\int_{-\infty}^{\infty} \frac {\log (x+i)} {(1+x^2)^n} dx
= 2\pi i \mathrm{Res} ( \log(z+i) (1+z^2)^{-n}; i ).
\end{align}
This yields
\begin{align}
\int_0^{\infty} 
\frac { \frac 12 \log(x^2+1) + i\theta_x} {(1+x^2)^n} dx
+ \int_0^{\infty} \frac {\frac 12 \log (x^2+1) +i(\pi -\theta_x) }
{ (1+x^2)^n} dx = 2\pi i \mathrm{Res} ( \log(z+i) (1+z^2)^{-n}; i ).
\end{align}
Thus
\begin{align}
F_n = 2 \int_0^{\infty} \frac {\log (x^2+1)} {(1+x^2)^n} dx
=4\pi \mathrm{Re} \Bigl( i \mathrm{Res} ( \log(z+i) (1+z^2)^{-n}; i )
\Bigr).
\end{align}
We obtain for $n=1,\cdots, 4$,
\begin{align}
&F_1 =\pi \log 4, \qquad F_2=\pi(-\frac 12 +\log 2); \\
&F_3=\pi(-\frac 7 {16}+\frac 34\log 2);\qquad F_4=
\pi(-\frac {37}{96}+\frac 58 \log 2).
\end{align}

Step 4. Verification of \eqref{A.102}.  Clearly
\begin{align}
\text{LHS of \eqref{A.102}}= 12 I_2 - 48 I_3 +48 I_4.
\end{align}
By using Step 1, we have
\begin{align}
I_{n+1} = (1-\frac 1{2n} )I_n - \frac 2n (F_n- F_{n+1}).
\end{align}
Clearly
\begin{align}
I_2=\frac 12 I_1 -2 (F_1- F_2) =
-\pi +\frac 1 6 \pi^3 +\frac 12 \pi (-2+\log 4) \log 4.
\end{align}
Similarly
\begin{align}
&I_3= \frac 1 {16} \pi (-11 + 2\pi^2
+ (\log 16) (-7+\log 64) ); \\
&I_4=\frac 1 {288}
\pi \Bigl(-155 + 30\pi^2 + 6 (\log 4) (-37+15 \log 4) \Bigr).
\end{align}
The identity \eqref{A.102} then follows easily.

\end{document}